\documentclass[12pt]{amsart}
\usepackage{amsthm}
\usepackage{amsmath, amsfonts, amsthm,amssymb, color, hyperref, enumerate, extarrows, cite, graphicx}
\usepackage{appendix}
\usepackage{enumitem}

\usepackage[margin=0.8in]{geometry}
\usepackage{physics}
\usepackage[dvipsnames]{xcolor}

\newtheorem{theorem}{Theorem}[section]

\newtheorem{definition}[theorem]{Definition}
\newtheorem{cor}[theorem]{Corollary}
\newtheorem{prop}[theorem]{Proposition}
\newtheorem{lem}[theorem]{Lemma}
\newtheorem*{claim*}{Claim}
\newtheorem*{notation*}{Notation}
{\theoremstyle{definition}

}
\newcommand{\R}{\mathbb{R}}
\newcommand{\C}{\mathbb{C}}
\newcommand{\T}{\mathbb{T}}
\newcommand{\Z}{\mathbb{Z}}

\numberwithin{equation}{section}
{\theoremstyle{definition}
\newtheorem{remark}[theorem]{Remark}
}
\usepackage[T1]{fontenc}

\usepackage[hyperpageref]{backref}
	
\hypersetup{
	colorlinks   = true,
	citecolor    = magenta}

\newcommand{\reg}{\hbox{reg}}
\newcommand{\sng}{\hbox{sng}}

\begin{document}
\title{Sharp decay rate for eigenfunctions of perturbed periodic Schrödinger operators}


\author[W.\ Liu]{Wencai Liu}
\email{liuwencai1226@gmail.com; wencail@tamu.edu}
\address{Department of Mathematics, Texas A\&M University, College Station TX, 77843, USA.}

\author[R.\ Matos]{Rodrigo Matos}
\email{rodrigo@mat.puc-rio.edu}
\address{Department of Mathematics, PUC-Rio, Rio de Janeiro RJ, 22451-900, Brazil.}

\author[J.\ N.\ Treuer]{John N. Treuer}
\email{jtreuer@ucsd.edu}
\address{Department of Mathematics, University of California San Diego, La Jolla, CA 92093, USA}

\maketitle

\begin{abstract}
This paper investigates uniqueness results for perturbed periodic Schr\"odinger operators on \(\mathbb{Z}^d\). Specifically, we consider operators of the form \(H = -\Delta + V + v\), where \(\Delta\) is the discrete Laplacian, \(V: \mathbb{Z}^d \rightarrow \mathbb{R}\) is a periodic potential, and \(v: \mathbb{Z}^d \rightarrow \mathbb{C}\) represents a decaying impurity. We establish quantitative conditions under which the equation \(-\Delta u + V u + v u = \lambda u\), for \(\lambda \in \mathbb{C}\), admits only the trivial solution \(u \equiv 0\). Key applications include the absence of embedded eigenvalues for operators with impurities decaying faster than any exponential function and the determination of sharp decay rates for eigenfunctions. Our findings extend previous works by providing precise decay conditions for impurities and analyzing different spectral regimes of \(\lambda\).  

\end{abstract}

\section{Introduction }

In this paper, we will concentrate on certain uniqueness results for finite difference equations arising from
locally perturbed periodic Schr\"odinger operators on $\Z^d$ of the type $H=-\Delta+V+v$, where $\Delta$ denotes the discrete Laplacian on $\mathbb{Z}^d$, $V:\mathbb{Z}^d\rightarrow \mathbb{R}$ is a periodic potential and $v:\mathbb{Z}^d\rightarrow \mathbb{C}$ is interpreted as an impurity which we will assume to be decaying ``sufficiently fast''.
Below we present several results which show that if $u\in \ell^2(\mathbb{Z}^d)$ solves
\begin{equation}\label{eq:eigenvalueeqH}-\Delta u+Vu+vu=\lambda u\,\,\,\,\text{in}\,\,\mathbb{Z}^d
\end{equation}
for some $\lambda\in \mathbb{C}$ and $v$ decays ``sufficiently fast'' then $u\equiv 0$. 
Our main contribution is that our results are quantitative. Namely, we provide specific conditions on how fast the impurity should decay for the uniqueness result to hold. One of our results, Theorem \ref{thm3:resolventdecay} below, is also shown to be sharp as $\lambda \to \infty$. The location of $\lambda$ will play a key role in the estimates of Theorems \ref{thm1:interiordecay}-\ref{thm3:resolventdecay}. Their statements differ according to the spectral regimes below: 
\begin{enumerate}[label=(\roman*)]
    \item \label{item:interiorcase} $\lambda$ belongs to the interior of a spectral band for $H_0:=-\Delta+V$.
    \item \label{item:endpointcase} $\lambda$ is an endpoint of a spectral band for $H_0$.
    \item \label{item:notess} $\lambda$ does not belong to the essential spectrum of $H$.
\end{enumerate}
Our setting is such that $\sigma_{\mathrm{ess}}(H)=\sigma(H_0)$ so the above regimes are mutually exclusive.  As it will become apparent throughout the paper, the proofs of the main results employ a combination of spectral tools along with complex analysis in several variables. For instance, in Theorems \ref{thm1:interiordecay} and \ref{thm2:endpointdecay}, we will make use of a Floquet transformation in order to rewrite the spectral problem and connect the decay of solutions to the analyticity of their Floquet transforms on a strip. We then apply recent developments in the discrete setting regarding the structure of the Fermi surface of $H_0$. In particular we utilize its irreducibility and related features, in order to check the assumptions of a result, originally due to Kuchment and Vainberg \cite{KuVa00}, regarding the cancellation of zeros of functions in several complex variables. 

Case \ref{item:interiorcase} is treated in Theorem \ref{thm1:interiordecay} where we prove that $H$ has no embedded eigenvalues if $v$ decays faster than any exponential (c.f. Definition \ref{def:betadecay}). This result already improves upon previous works. For instance, in \cite{IM14} the absence of embedded eigenvalues is shown without the presence of a periodic term, i.e.~if $V\equiv 0$, and assuming that $v$ is compactly supported. In \cite{LiuPreprint:Irreducibility}, the periodic term is considered but $v$ is required to decay super-exponentially with a precise estimate. We postpone further details to the discussion in Section \ref{subsec:settingsec} below Theorem \ref{thm1:interiordecay}. Absence of embedded eigenvalues for the discrete operator $H$ is mentioned as ``an interesting open question'' in \cite[Remark 6.6]{IsoKoro}, even without the presence of a periodic potential $V$. In this reference, trace formulas and inverse scattering are studied in the case where $V\equiv 0$ and $v$ is compactly supported. Below we highlight the main difficulties which are specific to the discrete setting. For now we stress that such uniqueness results are naturally expected to be more challenging in this context since equation \eqref{eq:eigenvalueeqH} carries less information than its continuous analogue.
Specifically, the lattice $\mathbb{Z}^d $ lacks the rotational symmetry of Euclidean space $\mathbb{R}^d.$

Case \ref{item:endpointcase} is the subject of Theorem \ref{thm2:endpointdecay}, where we show that if $\lambda$ is the endpoint of a spectral band for $H_0$ then $Hu=\lambda u$ does not admit exponentially decaying solutions if $v$ decays faster than any exponential. It is worth comparing and contrasting this result to the phenomenon of Anderson Localization (see \cite{A-W-B, Kirschsurv,Stolz-surv} and references therein) where exponentially decaying eigenfunctions occur at spectral edges in the presence of random, independent and identically distributed impurities.\par
In case \ref{item:notess},
 the distance between a complex surface $D(\lambda)$ and $\mathbb{R}^d$ plays an important role, see Remark \ref{rem:surfacedist} and Equation \eqref{eq:dispersion} for the definition of $D(\lambda)$. This surface is closely related to the Fermi surface at $\lambda$, although it is defined through different variables.
 In Theorem \ref{thm3:resolventdecay} we obtain a rate $r=r(\lambda)$ for which it is impossible to construct a solution to the finite difference equation \eqref{eq:eigenvalueeqH} that decays exponentially at a faster rate than $r$. In Theorem  \ref{thm3:resolventdecay} \ref{prop:optimal} we show with an example that this rate is, asymptotically speaking as $\lambda \to \infty$; it is the best one may find within the class of impurities $v$ considered here.
\subsection{ Brief review of the literature}
We now provide some context for this work, focusing on previous contributions which are technically close to the present paper. We start with the pioneering works \cite{KuVa00} and \cite{ShaVainb} which, to the best of our knowledge, set the stage for the applications of complex analysis in several variables to uniqueness questions such as the ones studied here.\par
In  \cite{KuVa00},
absence of embedded eigenvalues was shown  for continuous Schr\"odinger operators $-\Delta+V+v$, in dimensions two and three, with periodic potential $V$ and a ``sufficiently fast'' decaying perturbation $v$, provided
the corresponding Fermi surface is irreducible modulo natural symmetries. This work also introduced many ideas which are key to the present paper. Thus, we provide further comments below Theorem \ref{thm1:interiordecay} and also in Section \ref{Section: Kuchvainb}, which states the result of \cite[Lemma 17]{KuVa00} along with a local version of it which is needed for Theorem \ref{thm3:resolventdecay} below.
In \cite{ShaVainb} the authors studied the question of uniqueness of solutions of the equation
\begin{equation}\label{eq:eigenfshavai}
(-\Delta+q-\lambda)\psi=f\,\,\,\text{in}\,\,\,\mathbb{Z}^d
\end{equation} when both $q$ and $f$ have bounded support. In particular, they determined appropriate classes $W_{\pm}$ such that \eqref{eq:eigenfshavai} admits a unique solution $\psi_{\pm}$ within each of the classes  $W_{\pm}$ whenever $\lambda\in \sigma(-\Delta)\setminus S_0$, where $S_0$ is a finite set of `exceptional' points. Moreover, these solutions are obtained as limits of resolvent values $\mathcal{R}_{z}=(-\Delta-q-z)^{-1}f$ as $z\to \lambda\pm i0$, \cite[Theorem 10]{ShaVainb}.\par
In \cite{IM14}, the existence of `Rellich-type' solutions of \eqref{eq:eigenfshavai} was studied when $q\equiv  0$  and $f$ has compact support in $\mathbb{Z}^d$. Namely, let
$B_{R}=\{n\in \mathbb{Z}^d:\,\,\,|n|\leq R\}$, $B^c_{R}=\mathbb{Z}^d\setminus B_R$ and assume that $\psi$ satisfies, for some $R_0>0$ and $\lambda\in \sigma(-\Delta)$,
\begin{equation}
    (-\Delta-\lambda)\psi=0\,\,\,\text{in}\,\,\,B^c_{R_0}.
\end{equation}
Moreover suppose that $\psi$ decays according to
\begin{equation}\label{eq:decayRellich}
\lim_{R\to \infty} \sum_{n\in B_{R}\setminus B_{R_0}}\abs{\psi(n)}^2=0.
\end{equation}
It was proved in 
\cite[Theorem 1.1]{IM14} that there exists $R_1>R_0$ such that $\psi\equiv 0$ in $B^c_{R_1}$.
As a consequence, the authors derived a result on the absence of eigenvalues embedded into the essential spectrum (except for, possibly, its endpoints) for operators of the type $-\Delta+q$ where $q$ is compactly supported.\par
In \cite{AIM16}, a number of spectral questions were addressed for $-\Delta+v_0$, in various lattices, under the assumption that $v_0$ is finitely supported. For instance, the unique continuation property, absence of embedded eigenvalues and limiting absorption principle for the resolvent were treated there.
 In \cite{IMinv} the authors investigated the inverse scattering problem for the operator $-\Delta+v_0$ on $\mathbb{Z}^d$
where again $v_0$ is a finitely supported potential. There, it was shown that $v_0$ can be  uniquely reconstructed from a scattering matrix at a fixed energy within the spectrum of $-\Delta+v$ and away from an exceptional set. Inverse scattering for more general graphs was studied in \cite{KIM18}. \par
 More recently, absence of embedded eigenvalues for operators $-\Delta+V+v$, where $V$ is periodic and $v$ decays \emph{super}-exponentially was shown building on the proof of the irreducibility conjecture of Fermi varieties \cite{LiuPreprint:Irreducibility}.
In particular Theorem \ref{thm1:interiordecay} below improves upon this result. It is also worth mentioning that the irreducibility of Fermi varieties has been previously linked to the occurrence of embedded eigenvalues. Indeed, these connections were examined in earlier work by Kuchment-Vainberg \cite{KuVa00,kv06cmp} and Shipman\cite{shi1}.\par
The remainder of this paper is organized as follows: In Section \ref{subsec:settingsec} we state the main results of this work. In Section \ref{Section: Kuchvainb} we state a key result of several complex variables which is needed for the proofs of the main theorems. In Section \ref{Sec:Floquet} we present a few basic results in Floquet theory which shall be used later in the main proofs given in Sections \ref{Sec:proofs} and \ref{Sec:proofscontinued}. Various results of several complex variables are collected in the Appendix \ref{sec:proofovanishing}, which might be of independent interest.

\subsection{Setting and main results}
\label{subsec:settingsec}

	Given $q_i\in \Z_+$, $i=1,2,\ldots,d$,
	let $\Gamma=q_1\Z\oplus q_2 \Z\oplus\cdots\oplus q_d\Z$.
	We say that a function $V: \Z^d\to \R$ is  $\Gamma$-periodic (or just periodic)  if 
	for any $\gamma\in \Gamma$ and $n\in \mathbb{Z}^d $ the equality  $$V(n+\gamma)=V(n)$$ holds.

	Let  $\Delta$ be the adjacency operator on $\ell^2(\Z^d)$, namely
	\begin{equation*}
	(\Delta u)(n)=\sum_{|n^\prime-n|=1}u(n^\prime),
	\end{equation*}
	where $n=(n_1,n_2,\ldots,n_d)\in\Z^d$, $n^\prime=(n_1^\prime,n_2^\prime,\ldots,n_d^\prime)\in\Z^d$ and 
	\begin{equation*}
	|n^\prime-n|=\sum_{i=1}^d |n^\prime_i - n_i|.
	\end{equation*}
	Consider the  discrete Schr\"{o}dinger operator on $\ell^2({\Z}^d)$ given by
	\begin{equation} \label{h0}
	H_0=-\Delta +V ,
	\end{equation}
	where $V:\mathbb{Z}^d\rightarrow \mathbb{R}$ is $\Gamma$-periodic.
	
	In this paper, we always assume that $d\geq 1$, that the greatest common factor of $q_1,q_2,\ldots, q_d$ is 1 and that $V$ is $\Gamma$-periodic. Hence
	$H_0$ given by \eqref{h0} is a discrete periodic Schr\"{o}dinger operator.

 	Let $Q=q_1q_2\cdots q_d$.
	The spectrum of $H_0=-\Delta+V$, denoted by $\sigma (H_0)$, is the union of the spectral bands $[a_m,b_m]$, $m=1,2,\ldots, Q$, where each of them is the range of a band function $\lambda_m (k)$, $k\in[0,1)^d$, i.e:
	\begin{equation}\label{gband}
	\sigma (H_0)={\bigcup }_{m= 1}^{Q}[a_m,b_m].
	\end{equation}
	The main objects of our study are the perturbed periodic operators of the form
	\begin{equation}\label{Hop}
	H=H_0+v=-\Delta +V +v,   
	\end{equation}
where $v:\mathbb{Z}^d\rightarrow \mathbb{C}$ is a decaying function which is not identically zero. By Weyl's criterion, if $ \limsup_{|n|\to \infty} |v(n)|=0$ then 
	\begin{equation*}
	\sigma_{\rm ess}(H)=\sigma(H_0)={\bigcup }_{m = 1}^{Q}[a_m,b_m],
	\end{equation*}
	where $\sigma_{\rm ess}(H)$ is the essential spectrum of $H$. 
 Among other questions, we will be interested in the set of eigenvalues of $H$, denoted by $\sigma_p(H)$.

 \begin{definition}
 We say that the perturbed operator $H$ given by \eqref{Hop} does not have embedded eigenvalues if $\,\bigcup_{m=1}^{Q}(a_m,b_m)\cap \sigma_p(H)=\emptyset$.
 \end{definition}

 \begin{definition}\label{def:betadecay}
 Let $\beta>0$ be given.
     We say that $f:\mathbb{Z}^d\rightarrow \mathbb{C}$ has $\beta$-exponential decay
     if \begin{equation}\label{eq:betadecay}
     \limsup_{|n|\to \infty}\frac{\ln|f(n)|}{|n|}<-\beta.
     \end{equation}
 \end{definition}
 Our first result is the following

	\begin{theorem}\label{thm1:interiordecay}
		Let $d\geq 1$, $H$ be given by \eqref{Hop} and
		assume that $v$ has $\beta$-exponential decay for all $\beta>0$.
		Then $H$ does not have embedded eigenvalues.
	\end{theorem}

 The continuous version of Theorem \ref{thm1:interiordecay}, where $\Delta$ is replaced by the Laplacian on $\mathbb{R}^d$, was proven in \cite[Theorem 15]{KuVa00}, in dimension $d\leq 3$, under the following assumptions:
 \begin{enumerate}[label=(\alph*)]
      \item\label{intersecKV} Any irreducible component of the Fermi variety of $H_0$ intersects the real space $\mathbb{R}^d$ by a subset of dimension $d-1$.

\item \label{decayKV} The inequality \begin{equation} \abs{v(x)}\leq e^{-C\abs{x}^\gamma}\,\,\,\text{for some}\,\,\,\gamma>\frac{4}{3}
 \end{equation} 
 holds almost everywhere in $\mathbb{R}^d$.
 \end{enumerate}
 
  Condition \ref{decayKV} arises due to an application of the unique continuation result \cite{Bourgain-Kenig}. In dimension $d=2$, due to a result of Logunov-Malinnikova-Nadirashvili-Nazarov \cite{Log-Mal-Nad-Naz}, following the proof of \cite[Theorem 15]{KuVa00} one obtains the same result under the weaker assumption
 \begin{equation} \abs{v(x)}\leq e^{-C\abs{x}^\gamma}\,\,\,\text{for some}\,\,\,\gamma>1.
 \end{equation}
 On the other hand, in the discrete setting, the conclusion of Theorem \ref{thm1:interiordecay} was first established in \cite{LiuPreprint:Irreducibility} under the assumption that for all $n\in \mathbb{Z}^d$ \begin{equation} 
 \abs{v(n)}\leq e^{-C\abs{n}^\gamma}\,\,\,\text{for some}\,\,\,\gamma>1.
 \end{equation}
 For further discussions on the literature we refer the reader to Section 4 in \cite{liu2021topics}.
 Thus, the main technical contribution of Theorem \ref{thm1:interiordecay} is to replace the condition of super-exponential decay on $v$ by $\beta$-exponential decay for all $\beta>0$. As consequences, we obtain
 \begin{cor}\label{cor2:interiordecayLaplacian}
		Let $d\geq 1$ and assume that $v$ has $\beta$-exponential decay for all $\beta>0$. Then $\sigma_{\rm ess}(-\Delta+v)=[-2d,2d]$ and $\sigma_{\rm p}(-\Delta+v)\cap(-2d,2d)=\emptyset$.
	\end{cor}
\begin{remark}
    Under a stronger assumption that v has compact support, it was shown in \cite{IM14} that $\sigma_{\rm p}(-\Delta+v)\cap(-2d,2d)=\emptyset$ 
\end{remark} 
 \begin{remark}
 We stress that for our results the potential $v$ may be complex-valued whereas $V$ must be real-valued.
 \end{remark}

We now proceed to state our endpoint result.
\begin{theorem}\label{thm2:endpointdecay}
		Let $d\geq 1$ and $H$ be given by \eqref{Hop}.
		Assume that $v$ has $\beta$-exponential decay for all $\beta>0$,
  that 
  $Hu=\lambda u$ for some $\lambda \in \bigcup_{m=1}^{Q}[a_m,b_m]$ and that $u$ has $\beta_0$-exponential decay for \emph{some} $\beta_0>0$. Then $u\equiv0$.
	\end{theorem}
 An immediate consequence of theorem \ref{thm2:endpointdecay} is the following uniqueness result:
 \begin{cor} Let $d\geq 1$. Assume that $v$ has $\beta$-exponential decay for all $\beta>0$,
  that $u_1,u_2:\mathbb{Z}^d\rightarrow \mathbb{C}$ solve
  $$
  (-\Delta+V+v)u=\lambda u$$
 and that $u_1-u_2$ has $\beta_0$-exponential decay for \emph{some} $\beta_0>0$.
 Then $u_1\equiv u_2$.
 \end{cor}


 Finally, our third theorem below covers the case of solutions of the finite difference equation corresponding to energies away from $\sigma_{\mathrm{ess}}(H)$.

 \begin{theorem}\label{thm3:resolventdecay}
 \begin{enumerate} [label=(\Roman*)]
     \item \label{Ithm3:resolventdecay}
		Let $d\geq 1$, $H$ be given by \eqref{Hop} and 
		assume that $v$ has $\beta$-exponential decay for all $\beta>0$. Take
 $\lambda \notin\bigcup_{m=1}^{Q}[a_m,b_m]$. There exists $r=r(\lambda)$ such that 
		if  $Hu=\lambda u$  and
  $u$ has $\beta_0$-exponential decay for some $\beta_0>r(\lambda)$
  then $u\equiv0$.
  \item \label{asympthmii} Moreover, $r(\lambda)$ has the following asymptotic behavior \begin{equation}\label{eq:rasymp}\lim_{\abs{\lambda} \to \infty}\frac{r(\lambda)}{\log \abs{\lambda}}=1.
  \end{equation}

  \item \label{prop:optimal}
There exists a set $\mathcal{E}\subset \mathbb{R}$ of Lebesgue measure zero with the following property: \newline For all $\lambda \in \mathbb{C}\setminus(\mathcal{E}\cup \sigma(H_0))$ there exists $v=v_{\lambda}:\mathbb{Z}^d\rightarrow \mathbb{C}$  and $u=u_{\lambda}:\mathbb{Z}^d\rightarrow \mathbb{C}$ such that \eqref{eq:eigenvalueeqH} holds. Moreover $v$ has $\beta$-exponential decay for all $\beta>0$ and $u$ has $\beta_0$-exponential decay for all $\beta_0<\ln |\lambda |-\ln(2d)$.

  \end{enumerate}
	\end{theorem}
 \begin{remark} Combining items \ref{asympthmii} and \ref{prop:optimal} in the above theorem, one concludes that the rate $r(\lambda)$ is asymptotically sharp.
     
 \end{remark}
  \begin{remark}
     When $\lambda\in \mathbb{R}\setminus (\mathcal{E}\cup \sigma(H_0))$ it follows from the proof that $v_\lambda$ in item \ref{prop:optimal} is real-valued.
 \end{remark}
 
 \begin{remark}\label{rem:surfacedist}

 It follows from our arguments that
     $$r=r(\lambda):=2\pi\mathrm{dist}\left(\mathbb{R}^d,D(\lambda)\right)$$
 where $D(\lambda):=\{x\in \mathbb{C}^d:\,\,\det({\tilde{H}}_0(x)-\lambda I)=0\}$ and ${\tilde{H}}_0(x)$ are the fiber operators obtained from $H_0$ after a Floquet decomposition. See Sections \ref{Sec:Floquet} and \ref{Sec:proofs} for further details.
 \end{remark}

 	\begin{cor}\label{cor3:resolventdecayLaplacian}
   Let $d\geq 1$ and
 $\lambda \notin[-2d,2d]$.
		Assume that $v$ has $\beta$-exponential decay for all $\beta>0$. 
		If  $(-\Delta+v)u=\lambda u$ and $u$ has $\beta$-exponential decay for some $\beta>r(\lambda)$,
 then $u\equiv0$.
	\end{cor}

\section{The  Kuchment-Vainberg Theorem}\label{Section: Kuchvainb}

A subset $A$ of a complex manifold $\Omega$ is called an analytic subset if it is locally the zero set of finitely many holomorphic functions.  That is, for all $z \in \Omega$, there exist a neighborhood $U$ of $z$ and holomorphic functions $f_1,\ldots, f_N$ defined on $U$ such that
$$
\{z \in U: f_1(z) = \cdots = f_N(z) = 0\} = U \cap A
$$

Analytic subsets are closed in $\Omega$ and the intersection of finitely many analytic subsets is again an analytic subset.  An analytic subset $A\subset \Omega$ is reducible if it can be written as a nontrivial union of two analytic subsets. Otherwise, $A$ called irreducible.  All analytic subsets have a decomposition into irreducible components; that is, there is a countable index set $I$ such that $A = \cup_{j \in I} Z_j$ where each $Z_j$ is an irreducible analytic subset of $\Omega$. The regular points of $A$ are the points where locally $A$ is a complex manifold. The set of regular points of $A$, denoted henceforth by $\reg(A)$, forms an open, dense subset of $A$. We write $\mathrm{sng } A=A\setminus \reg(A)$ for the set of singular points of $A$.  The dimension of $A$ at a regular point $p\in A$ is the (complex) dimension of the manifold locally equal to $A$ near $p$.  The dimension of a singular point $p$ is the limit supremum of the dimension of a regular point $z$ as $z$ approaches $p$. We refer the reader to \cite{Ch89} for background on analytic subsets as well as to Appendix \ref{The appendix} below where additional properties of analytic subsets will be developed. In what follows we denote by $\mathcal{O}(\Omega)$ the ring of holomorphic functions in $\Omega$.

\begin{theorem}\label{Kuch-Vainb theorem}\cite[Lemma 17]{KuVa00}
Suppose $\Omega \subset \mathbb{C}^d$ is a domain. Let $A = \{z \in \Omega: h(z) = 0\}$ with $h\in \mathcal{O}(\Omega)$ and denote by $A = \cup_{j \in I} Z_j$ its decomposition into irreducible components. Let $g \in \mathcal{O}(\Omega)$ and define $f = g/h$. Suppose that one of the following statements holds
\begin{enumerate} [label=(\roman*)]
\item \label{condition1} $f  \in L^2_{loc}(\mathbb{R}^d)$ and for all $j$,  $Z_j \cap \mathbb{R}^d$ is nonempty and  has Hausdorff dimension $d - 1$.

\item \label{condition2} $f\in O(\Omega^{\prime})$ for some domain  $\Omega^{\prime}\subset\Omega$ such that for all $j\in I$ $Z_j \cap \Omega^{\prime}\neq \emptyset$.
    
\end{enumerate}
Then $f$ extends to a holomorphic function on $\Omega$.
\end{theorem}
 Part \ref{condition1} has slightly weaker hypotheses than \cite[Lemma 17]{KuVa00}, where the same conclusion is reached under the assumption that for each $j$, $Z_j \cap \mathbb{R}^d$ contains a smooth submanifold of real dimension $d-1$. For now we mention that the improvement presented in the above version relies on a stratification result of Whitney which may be found in \cite[Theorem 1.2.10]{Trotman}. Further details are postponed until Lemma \ref{C1 manifold leads to real-analytic manifold lemma} in the Appendix  \ref {sec:proofovanishing} below. Since we will also need the second statement for Theorem \ref{thm3:resolventdecay}
 we provide a proof of both of these results for completeness. Before doing so, we introduce an auxiliary lemma for convenience. 

\begin{lem}\label{Lem:vanishing} For $d > 1$, in the setting of Theorem \ref{Kuch-Vainb theorem}, for each $j\in I$, let $m_j$ be the minimal order of the zeros of $h$ on $Z_j$. If either \ref{condition1} or \ref{condition2} from Theorem \ref{Kuch-Vainb theorem} holds, then there is a $k_j \geq m_j$ such that $g$ vanishes to order at least $k_j$ on an open subset $V_j$ of $Z_j$.
    
\end{lem}
See Definition \ref{Definition in the appendix of what the order is and at least order} for the definition of order.  The proof of this technical lemma is postponed to Section 
\ref {sec:proofovanishing}.

\begin{proof}[Proof of Theorem \ref{Kuch-Vainb theorem}]    We first prove the $d = 1$ case, which is considerably simpler than the $d > 1$ cases.  When $d = 1$, $A$ is discrete and each $Z_j$ is a singleton.  Thus in \ref{condition1}, $A \subset \mathbb{R}$.  In a neighborhood of $x_0 \in A$, $f(z) = q(z)(z-x_0)^{k}$ for a nonzero holomorphic function near $x_0$ and integer $k$.  Since $f$ is square-integrable on an interval containing $x_0$, $k \geq 0$. Thus, the conclusion follows.  For \ref{condition2}, the conclusion follows after noticing that $A \subset \Omega^{\prime}$.

We now prove the $d > 1$ cases.  Since \ref{condition1} or \ref{condition2} holds, $h$ is not identically zero.  Thus, the dimension of $A$ at any of its points is equal to $d-1$, \cite[p.2.6.~Proposition]{Ch89} and, in particular, $\dim(Z_j) = d-1$.  Let $m_j$ be the minimal order of the zeros of $h$ on $Z_j$. By assumption we have that  $m_j \geq 1$. Consider the following analytic subset of $\Omega$:
    \begin{equation}\label{definition of Aj prime}
    A_j^{\prime} = \left\{z \in \Omega: {\partial^{\alpha}h \over \partial z^{\alpha}}(z) = 0, \quad 0 \leq |\alpha| \leq m_j\right\} \cap Z_j.
    \end{equation}
    We note that $A_j^{\prime}$ may possibly be empty.
         By \cite[p.5.3.~Corollary 1]{Ch89}, since $A_j^{\prime} \subsetneq Z_j$  with $Z_j$ irreducible, we have that $A_j^{\prime}$ is nowhere dense in $Z_j$ and, moreover, 
    \begin{equation}\label{dimension of Ajprime}
        \dim(A^{\prime}_j) < \dim(Z_j) = d - 1.
    \end{equation}
    Additionally, $Z_j \setminus A_j^{\prime}$ is open in $Z_j$.  Note that by Lemma \ref{Lem:vanishing}, $g$ vanishes to order at least $k_j \geq m_j$ on an open subset $V_j$ of $Z_j$.

    We now conclude the proof for both cases simultaneously. Let 
    $$
    B = \left\{z \in \Omega: {\partial^{\alpha} g \over \partial z^{\alpha}}(z) = 0, |\alpha| < k_j\right\}.
    $$
    $B$ is an analytic subset of $\Omega$ such that $V_j \subset B\cap Z_j$. Since $Z_j$ is irreducible we have that $Z_j \subset B$, cf. \cite[p.5.3.~Corollary 2]{Ch89}. Since  $Z_j \setminus \left(\sng(A) \cup A_j^{\prime}\right)\subset \reg(A)$, given $p \in Z_j \setminus \left(\sng(A) \cup A_j^{\prime}\right)$ there is an $\epsilon > 0$ such that
    $$
    \mathbb{B}(p, \epsilon) \cap A = \mathbb{B}(p, \epsilon) \cap Z_j
    $$
    is a complex manifold, see Lemma \ref{lemma that involves local finiteness and irreducibility} below. Here $\mathbb{B}(p, \epsilon)$ is the ball centered at $p$ of radius $\epsilon$.  The set $A_j^{\prime} \cup \left(\sng(A) \cap Z_j\right)$ is a closed subset of $Z_j$.  Thus, after possibly shrinking $\epsilon$, we see that
    $$
    \mathbb{B}(p, \epsilon) \cap Z_j \subset Z_j \setminus \left( A_j^{\prime} \cup \sng(A)\right) \subset Z_j \setminus A_j^{\prime},
    $$
    as well.  In summary, $\mathbb{B}(p, \epsilon) \cap Z_j$ is a complex manifold where $h(z)$ vanishes to order $m_j$ on it and $g(z)$ vanishes to order at least $k_j \geq m_j$ on it. After a holomorphic change of coordinates on $\mathbb{B}(p, \epsilon) \cap Z_j$ we see that
    $$
    f = {g \over h} = z_d^{k_j - m_j}q(z),
    $$
    where $q$ is holomorphic and $k_j \geq m_j$, see Lemma \ref{lemma involving weierstrass division theorem} below. Thus $f$ initially defined on $\mathbb{B}(p, \epsilon) \setminus A = \mathbb{B}(p, \epsilon) \setminus Z_j$ has an analytic continuation to $\mathbb{B}(p, \epsilon)$ and hence to $\Omega \setminus \left(\cup_{j \in I} A_j^{\prime} \cup \sng(A)\right)$ as well.  The union is countable and the Hausdorff dimension of  each of the sets $\{A_j^{\prime}\}_{j\in I}$, $\sng(A)$ is at most $2d - 4$. It follows that $f$ extends to a holomorphic function on $\Omega$, cf. \cite[Appendix A.1.3.~Proposition 3]{Ch89}.
    
\end{proof}

\section{Basics of Floquet-Theory}
\label{Sec:Floquet}
Before presenting the proofs of Theorems \ref{thm1:interiordecay}-\ref{thm3:resolventdecay} we recall some basic facts regarding various unitary transformations which will be used in the sequel. For a more detailed presentation we refer to \cite[Section 5]{FLM1}.
Let
 \begin{equation}
 W=\mathbb{Z}^d\cap \prod^{d}_{j=1}[0,q_j)
 \end{equation}
 be a fundamental domain for the potential $V$ and define
 
	\begin{equation*}
	W^{\ast}=\left\{0,\frac{1}{q_1},\frac{2}{q_1},\cdots,\frac{q_1-1}{q_1}\right\}\times\cdots \times \left\{0,\frac{1}{q_d},\frac{2}{q_d},\cdots,\frac{q_d-1}{q_d}\right\}\subset [0,1]^d.
	\end{equation*}
	The discrete Fourier transform of $V$ is  
	\begin{equation*}
	\widehat{V}(l) =\frac{1}{\sqrt{Q}}\sum_{ n\in {W} } V(n) e^{-2\pi i l\cdot n},\,\,\,l\in W^{\ast}.
	\end{equation*}
	Here, $l\cdot n= \sum_{j=1}^d l_j n_j$ for $l=(l_1,l_2,\ldots,l_d)\in W^{\ast}$ and $n=(n_1,n_2,\ldots,n_d)\in \Z^d$.
	For convenience, we  extend  $\widehat{V}(l)$ to  $W^{\ast}+\Z^d$ periodically. Namely, for all $l\in W^{\ast}$ and $n\in \Z^d$ we let
	\begin{equation*}
	\widehat{V}(l+n)=\widehat{V}(l).
	\end{equation*}
	Then, the inversion formula
	\begin{equation*}
	V(n)=\frac{1}{\sqrt{Q}}\sum_{l \in W^{\ast}}\widehat{V}(l) e^{2\pi il\cdot n}
	\end{equation*}
 holds true for any $n\in \Z^d$.
	For a function $u\in\ell^1(\Z^d)$, its Fourier transform defined on the torus $\T^d:=\R^d/\Z^d$ is
	$\mathcal{F}(u) :\T^d\to\C$ is given by
	\begin{equation*}
	\mathcal{F}(u)(x)=\sum_{n\in \Z^d}u(n)e^{-2\pi i n\cdot x}.
	\end{equation*}
 The following Plancherel type identity holds whenever $u_1,u_2\in \ell^1(\Z^d)$:
 \begin{equation}\label{eq:Plancherel}
 \langle \mathcal{F}(u_1), \mathcal{F}(u_2) \rangle_{L^2(\T^d)}=\langle u_1, u_2 \rangle_{\ell^2(\Z^d)}.
 \end{equation}
 This enables the extension of $\mathcal{F}$ to $\ell^2(\Z^d)$.
 We will often write $\hat{u}$ to denote $\mathcal{F}(u)$ as well. 
	
	For any $q$-periodic function $V:\mathbb{Z}^d\rightarrow \mathbb{C}$   and $u\in \ell^2(\Z^d)$, one has that
	\begin{equation*}
	\widehat{Vu}(x)=\frac{1}{\sqrt{Q}}\sum_{l\in W^{\ast} } \widehat{V}(l)\hat{u}(x-l),
	\end{equation*}
 see \cite[Proposition 5.1]{FLM1}.
 Let us now define $\mathbb{T}^d_* = \mathbb{R}^d/\Gamma^*$,
\begin{equation*}\mathcal{H}_q =  \int_{\mathbb{T}^d_*}^\oplus \mathbb{C}^W \, \frac{dx}{|\mathbb{T}^d_*|} = L^2\left(\mathbb{T}^d_*,\mathbb{C}^W; \frac{dx}{|\mathbb{T}^d_*|}\right) 
\end{equation*}
and $\mathcal{F}_q : \ell^2(\mathbb{Z}^d) \to \mathcal{H}_q$ given by
\begin{equation}\label{eq:Flouquettransf}(\mathcal{F}_q(u))(x,j)= \widehat{u}(x,j)=\widehat{u}_j(x) := \sum_{n \in \mathbb{Z}^d} e^{-2 \pi i \langle n \odot q,x\rangle}u(j+n\odot q), \ x \in \mathbb{T}^d_*, \ j \in W,
\end{equation}
where $a\odot b = (a_1b_1,\ldots,a_db_d)$ for ordered $d$-tuples $a=(a_1,\ldots,a_d)$ and $b = (b_1,\ldots,b_d)$.
As usual, $\mathcal{F}_q$ is initially defined for $\ell^1$ vectors, but  has a unique extension to a unitary operator on $\ell^2$ via Plancherel's identity.

\begin{lem}\cite[Proposition 5.2]{FLM1} \label{lem:directint}
The operator $\mathcal{F}_q$ is unitary. If $V$ is $q$-periodic, then 
\[
\mathcal{F}_q H_0 \mathcal{F}_q^* = \int^\oplus_{\mathbb{T}^d_*} \widetilde{H}_{0}(x) \, \frac{dx}{|\mathbb{T}^d_*|},
\]
where $ \widetilde{H}_{0}(x)$ denotes the restriction of $H_0$ to $W$ with boundary conditions 
\begin{equation}\label{g1}
u(n+m\odot q)  = e^{2\pi i \langle m\odot q,x\rangle } u(n), \quad  n,m\in\Z^d.
\end{equation}
\end{lem}

Given $x \in \mathbb{R}^d$, let $\mathcal{F}^{x}$ be the Floquet-Bloch transform defined on $\mathbb{C}^{\#W}$ as follows: for any vector on $W$, $\{u(n)\}_{n\in W}$, we set
\[ [\mathcal{F}^{x} u](l) 
= \frac{1}{\sqrt{Q}} \sum_{n \in W} e^{-2\pi i  \sum_{j=1}^d  \left(\frac{l_j}{q_j}+x_j\right) n_j }u(n), \quad l\in W .\]

Therefore,
\[ [(\mathcal{F}^{x} )^*u](l)
= \frac{1}{\sqrt{Q}} \sum_{n \in W} e^{2\pi i  \sum_{j=1}^d  \left(\frac{n_j}{q_j}+x_j\right) l_j }u(n), \quad l\in W .\]
Letting $z_j=e^{2\pi i x_j}$ we have the following:
\begin{lem}\cite[Proposition 5.3]{FLM1} \label{prop:floquetTransf}
Assume that $V$ is $q$-periodic.
Then 
$\widetilde{H}_0(x)$ given by \eqref{g1}  is unitarily  equivalent to 	
$
D^x+B_V,
$
where $D^x$ is a diagonal matrix with entries
\begin{equation}\label{D}
D^x(n,n^\prime) = -\sum^{d}_{j=1}\left(\mu^{j}_{n}e^{2\pi i x_j}+\frac{1}{\mu^{j}_{n}e^{2\pi i x_j}}\right)\delta_{n,n^{\prime}},
\end{equation}
where
$\mu^{j}_{n}=e^{2\pi i\frac{n_j}{q_j}}$ , and $B=B_V$ has entries related to the discrete Fourier transform of $V$ via
\begin{equation}\label{B} B(n,n')=\widehat V\left(\frac{n_1-n'_1}{q_1},\ldots,\frac{n_{d}-n'_d}{q_d}\right).
\end{equation}
\end{lem}

	

 \section{Proofs of Theorems \ref{thm1:interiordecay}-\ref{thm3:resolventdecay} \ref{Ithm3:resolventdecay}}\label{Sec:proofs}
 In order to prove the main theorems of this note we will make use of a well-known connection between the decay of Fourier coefficients and the regularity of functions $\hat{f} \in L^2(\mathbb{T}^d)$. Since we will need quantitative estimates we provide all of the details for completeness.
 
\subsection{Analyticity in a strip and relationship with the Fourier coefficients}

For $\rho > 0$, we let $\mathcal{S}_{\rho} = \{z=(z_1,\ldots,z_k): \max_{k} |\hbox{Im}\, z_k| < \rho\}$.  Below we denote $\|n\|_{\max} = \max \{\,|n_j|:\,\,j=1,\ldots, d\}$. 

\begin{lem}\label{Lem:relationship analiticitydecay}
    Let $\hat{f} \in L^2(\mathbb{T}^d)$ and $\{f_n\}$ be its Fourier coefficients; that is,
    $$
    f_n = \int_{\mathbb{T}^d}\hat{f}(x)e^{-2\pi in\cdot x}dx,\quad n \in \mathbb{Z}^d,
    $$
    where $n \cdot x = \sum_{k=1}^d n_kx_k$.
    \begin{enumerate}[label=(\alph*)]
        \item\label{decayFourierimpliesanaliticity} Suppose that $|f_n| \leq C_1e^{-C_2\|n\|_{\max}}$ for some constants $C_1, C_2 > 0$.  Then $\hat{f}$ is analytic in $\mathcal{S}_{C_2 \over 2\pi}$.

        \item \label{analiticityimpliesdecay} Suppose that $\hat{f}$ is analytic in $\mathcal{S}_{C_2}$.  Then for any $\epsilon > 0$,
        $$
        |f_n| \leq C_3\left\|\hat{f}\right\|_{L^\infty(\mathcal{S}_{C_2 - \epsilon})}e^{-2\pi\|n\|_{\max}(C_2 - \epsilon)}, \quad n \in \mathbb{Z}^d,
        $$
        where $C_3$ is a numerical constant independent of $C_2, n$ and $\hat{f}$.
    \end{enumerate}
\end{lem}
\begin{proof}
    To prove part 1, it suffices to show that
    $$
    \hat{f}(z) = \sum_{n \in \mathbb{Z}^d}f_ne^{2\pi i n \cdot z}
    $$
    converges uniformly on compact subsets of $S_{C_2 \over 2\pi}$.  Notice that 
    \begin{eqnarray}
        \sum_{n \in \mathbb{Z}^d} | \hat{f}(z) | 
        &\leq C_1\sum_{l=0}^\infty\sum_{\{n: \|n\|_{\max} = l\}}e^{\|n\|_{\max}\left(-C_2+ 2\pi|\mathrm{Im}\,z|\right)} \nonumber
        \\
        &\leq C_1\sum_{l=0}^{\infty} 2d(2l+1)^{d-1}e^{l\left(-C_2+ 2\pi|\mathrm{Im}\,z|\right)}, \label{this series converges uniformly on compact sets}
    \end{eqnarray}
    where we used that 
    $$
    \left|\{n: \|n\|_{\max} = l\}\right| \leq 2d(2l + 1)^{d-1}.
    $$
    The series \eqref{this series converges uniformly on compact sets} converges uniformly in any compact subset of $S_{C_2 \over 2\pi}$, which concludes the proof of \ref{decayFourierimpliesanaliticity}.

    To prove \ref{analiticityimpliesdecay}, start by letting $R =  C_2 - \epsilon$, where $C_2$ is as in \ref{decayFourierimpliesanaliticity}, and denote
    $$
    \gamma_1 = [0, 1], \quad \gamma_2 = [1, 1 - iR], \quad \gamma_3 = [1 - iR, -iR], \quad \gamma_4 = [-iR, 0],
    $$
    and for any $j \in \{1,\ldots,d\}$, let $z_1,\ldots, z_{j-1}, z_{j+1},\ldots, z_d \in [0, 1]$. Since $\hat{f}(z+e_j)=\hat{f}(z)$ for any $j=1,\ldots,d$ we have that
    \begin{equation*}\label{using periodicity assumption}
        \int_{\gamma_2} \hat{f}(z_1,\ldots, z_d)e^{-2\pi i n_jz_j}dz_j = -\int_{\gamma_4} \hat{f}(z_1,\ldots, z_d)e^{-2\pi i n_jz_j}dz_j.
    \end{equation*}
    It then follows by the analyticity of $\hat{f}$ on $S_{C_2}$ that

    \begin{equation}\label{eq:movecontour}
    \int_{\gamma_1} \hat{f}(z)e^{-2\pi i n_j z_j}dz_j = \int_{-\gamma_3}\hat{f}(z)e^{-2\pi i n_j z_j}dz_j.
    \end{equation}
    Fix $n \in \mathbb{Z}^d$.  Without loss of generality (after possibly reordering the coordinates) we may suppose that $|n_1| = \|n\|_{\max}$.  We may also assume that $n_1 \geq 0$.  (When $n_1 < 0$, to modify the proof one changes $R$ to $-R$ in the contour integrals $\gamma_2$, $\gamma_3$ and $\gamma_4$.)  Notice that by \eqref{eq:movecontour} and the choice of $n_1$
    \begin{align*}
        |f_n| &= \left|\int_{\mathbb{T}^{d-1}}e^{-2\pi i \sum_{j=2}^d n_jx_j}dx_2\cdots dx_d\int_{[-iR, 1 - iR]} \hat{f}(z_1,x_2,\ldots,x_d)e^{-2\pi in_1z_1}dz_1\right|\\
        &\leq C_3\|\hat{f}\|_{L^\infty(\mathcal{S}_{C_2 - \epsilon})}e^{-2\pi\|n\|_{\max}(C_2 - \epsilon)},
        \end{align*}
        which concludes the proof.
\end{proof}

\begin{cor}
  Let $u\in \ell^2\left(\mathbb{Z}^d\right)$ and let $\hat{u}\in L^2\left(\mathbb{T}^d_*,\mathbb{C}^W; \frac{dx}{|\mathbb{T}^d_*|}\right)$ be its Floquet transform defined by \eqref{eq:Flouquettransf}.

    \begin{enumerate}[label=(\alph*)]
        \item\label{eq:utouhat} Suppose that $|u(n)| \leq C_1e^{-C_2\|n\|_{\max}}$ for some constants $C_1, C_2 > 0$.  Then $\hat{u}(\cdot,j)$ is analytic in $\mathcal{S}_{C_2 \over 2\pi}$ for each $j\in W$.

        \item \label{eq:uhattou} Suppose that $\hat{u}(\cdot,j)$ is analytic in $\mathcal{S}_{C_2}$ for each $j\in W$.  Then for any $\epsilon > 0$ and  $n \in \mathbb{Z}^d$
        $$
        |u(n)| \leq C_3\max_{j\in W}\left\|\hat{u}(\cdot,j)\right\|_{L^\infty(\mathcal{S}_{C_2 - \epsilon})}e^{-2\pi\|n-\ell(n)\|_{\max}(C_2 - \epsilon)},
        $$
        where $C_3=C_3(d,q)>0$ and $\ell(n)$ is the only vector in $W$ satisfying
        $n=m\odot q +\ell(n)$ for some $m\in \mathbb{Z}^d$.
    \end{enumerate}
 \end{cor}

 \subsection{Proofs of Theorems \ref{thm1:interiordecay}-\ref{thm3:resolventdecay} \ref{Ithm3:resolventdecay}}
	
 We now introduce a calculation which will be the starting point for the proofs of Theorems \ref{thm1:interiordecay}-\ref{thm3:resolventdecay}. Below we denote by $\tilde{M}(x,\lambda)$ the  $Q\times Q$ matrix corresponding to the operator $\tilde{H}_0(x)-\lambda I$, by $\tilde{B}(x,\lambda)$ its adjoint and define
  \begin{equation}
 \tilde{P}(x,\lambda):=\det \tilde{M}(x,\lambda).
  \end{equation}
\begin{lem}\label{Lem:fractionform} Suppose that there exists $\lambda\in \sigma_{p}(H)$ and let $u\in \ell^2(\Z^d)\setminus\{0\}$ be an associated eigenfunction:  $$-\Delta u+Vu+vu=\lambda u.$$
Then, defining $\psi=-vu$ we have that for almost every $x\in \mathbb{T}^d_*$  \begin{equation}\label{eq:haturepres}
		\hat{u}(x,\cdot)=\frac{\tilde{B}(x,\lambda)\hat{\psi}(x,\cdot)}{\tilde{P} (x,\lambda)},\,\,\, \text{in}\,\,\mathcal{H}_q.
		\end{equation}

\end{lem}
\begin{proof}

With the given definitions, from the eigenfunction equation we have that \begin{equation}\label{geigen1}
		(H_0-\lambda I)u=\psi.
		\end{equation}
		Applying $\mathcal{F}_q$ to both sides of  \eqref{geigen1}, using Lemma \ref{lem:directint} and denoting $\mathcal{F}_q u=\hat{u}$ we find that for each $x\in \mathbb{T}^d_*$ 
\begin{equation}\label{eq:contradicstart}
(\tilde{H}_0(x)-\lambda I)\hat{u}(x,\cdot)=\hat{\psi}(x,\cdot)
		\end{equation}
  with both vectors above viewed as elements of $\mathcal{H}_q$.
  
  By Cramer's rule, if $x\in \mathbb{T}^d_* $ is such that $\lambda \notin \sigma (\tilde{H}_0(x))$ then
		\begin{equation*}
		(\tilde{H}_0(x)-\lambda I)^{-1}=\frac{\tilde{B}(x,\lambda)}{\tilde{P} (x,\lambda)}.
		\end{equation*}

		It follows from \eqref{eq:contradicstart} that whenever $\det \tilde{M}(x,\lambda)\neq 0$ 
		\begin{equation}
		\hat{u}(x,\cdot)=\frac{\tilde{B}(x,\lambda)\hat{\psi}(x,\cdot)}{\tilde{P} (x,\lambda)},\,\,\, \text{in}\,\,\mathcal{H}_q.
		\end{equation}
  In particular, the equality \eqref{eq:haturepres} holds true for almost every $x\in \mathbb{T}^d_*$ since the analytic set 
  \begin{equation}\label{eq:realzeroset}
  \{x\in \mathbb{R}^d:\,\,\tilde{P}(x,\lambda)=0\}
  \end{equation}
  has dimension at most $d-1$,  see \cite[Theorem 1.2.10]{Trotman}.
  \end{proof}

 The following result will be useful in order to check the first assumption of Theorem \ref{Kuch-Vainb theorem}. Before stating it, we recall the definition of the Fermi variety.\par
\begin{definition}Let $z\in \mathbb{C}$, $q=(q_1,\ldots,q_d)$ and denote by $\mathcal{H}(z,q)$ the space of those $\psi :\mathbb{Z}^d \to \mathbb{C}$ for which
\begin{equation}
\psi(n+j\odot q) = z^j \psi(n), \quad \forall \,n,j\in \mathbb{Z}^d.
\end{equation}
The (complex) Fermi variety of $H_0=-\Delta+V$ at $\lambda\in \mathbb{C}$ is 
\begin{equation} \label{eq:blochvarDef}
F_{\lambda}(V)= \{k \in \mathbb{C}^{d} : H_0\psi = \lambda \psi \text{ enjoys a nonzero solution in } \mathcal{H}(e^{2\pi i k},q)\},
\end{equation}
where $e^{2\pi i k} = (e^{2\pi i k_1},\ldots,e^{2\pi i k_d}).$
    
\end{definition} 

  \begin{lem}\label{lem:dimensioncheck}
       Define $$D(\lambda):=\{x\in \mathbb{C}^d:\,\,{\tilde P}(x,\lambda)=0\}$$ and let $Z_j$ be any of its irreducible components. Assume that  $\lambda\in (a_m,b_m)\cap \sigma_{p}(H)$ for some $m\in\{1,\ldots,Q\}$. Then $Z_j\cap \mathbb{R}^d$ is non-empty and has Hausdorff dimension $d-1$.
  \end{lem}
  \begin{proof} When $d\geq 3$, (respectively $d=2$ and $\lambda \neq [V]$) by \cite[Theorem 1.1]{LiuPreprint:Irreducibility} (respectively \cite[Theorem 1.2]{LiuPreprint:Irreducibility}) the Fermi variety $F_{\lambda}(V)$ is irreducible. Therefore, in the above cases, the result of \cite[Lemma 10]{KuVa00} applies and yields that
  $F_{\lambda}(V)\cap \mathbb{R}^d$ has Hausdorff dimension $d-1$. Since $(k,\lambda)\in F_{\lambda}(V)$ if and only if $k_j=\frac{x_j}{q_j}$ for some $x=(x_1,\ldots,x_d)\in D(\lambda)$, the desired result then follows. When $d=2$ and $\lambda= [V]$, $F_{\lambda}(V)$ may possibly be reducible so the above argument does not apply. However, in this case, by \cite[Remark 4]{LiuPreprint:Irreducibility}, $F_{\lambda}(V)=F_{0}(\bf{0})$ and the conclusion follows from a direct calculation.
When $d=1$, by standard Floquet theory (e.g. \cite[Theorem 1.14.]{Kuchment2016BAMS}), $D(\lambda)\subset \R$. Thus $Z_j\cap \mathbb{R}$ is non-empty.

  \end{proof}
 
 	\begin{proof}[\bf Proof of Theorem \ref{thm1:interiordecay}]

	Assume, for the sake of contradiction that, $\sigma_{p}(H)\cap \cup^{Q}_{m=1}(a_m,b_m)\neq \emptyset$ and take $\lambda\in (a_m,b_m)\cap \sigma_{p}(H)$ for some $m\in\{1,\ldots,Q\}$ with a corresponding eigenfunction $u\in \ell^2(\mathbb{Z}^d)$. By Lemma \ref{lem:dimensioncheck}, $\zeta(x)=\tilde{P} (x,\lambda)$  satisfies the dimension condition in \ref{condition1} of Theorem \ref{Kuch-Vainb theorem}.
		Since, by definition of $\mathcal{F}_q$, $\hat{u}\in L^2\left(\mathbb{T}^d_*,\mathbb{C}^W; \frac{dx}{|\mathbb{T}^d_*|}\right)$ we have that for any fixed $l\in W$ the map $x\mapsto \hat{u}(x,l)$ is an element of $L^2(\mathbb{T}^d_*)$. Moreover, this map is $\Gamma^\ast$ periodic by definition thus it may be viewed as a function $\hat{u}(\cdot,l)\in L^2_{\mathrm{loc}}(\mathbb{R}^d)$. Therefore, by Theorem \ref{Kuch-Vainb theorem} \ref{condition1} and \eqref{eq:haturepres}, one has that for each $l\in W$, $x\mapsto \hat{u}(x,l)$ is an entire function. It then follows from Lemma \ref{Lem:relationship analiticitydecay} that for any $\beta>0$ there exists $C_{\beta}$ with
  $$\abs{u(n)}\leq C_{\beta}e^{-\beta\abs{n}}$$
  which in turn contradicts the weak form of unique continuation in \cite[P.49]{bk13} or \cite{LyMa18}.
  \end{proof}

  \begin{proof}[\bf Proof of Theorems \ref{thm2:endpointdecay} and \ref{thm3:resolventdecay} \ref{Ithm3:resolventdecay}]
  We again start from \eqref{eq:haturepres}. Note that in Theorem \ref{thm2:endpointdecay} we allow $\lambda$ to take extreme values within each band of $\sigma(H_0)$, namely $\lambda=a_m$ and $\lambda=b_m$ for some $m\in\{1,\ldots,Q\}$. In this case the assumption \ref{condition1} in Theorem \ref{Kuch-Vainb theorem} is not necessarily valid, so we must proceed differently from the proof of Theorem \ref{thm1:interiordecay} above. Nonetheless, due to the assumption that
  $$\limsup_{\abs{n}\to \infty}\frac{\ln\abs{u(n)}}{\abs{n}}<-r,$$
  where $r=0$ for Theorem \ref{thm2:endpointdecay} and, more generally,
  \begin{equation}\label{eq:defdistance}
r=r(\lambda):=2\pi\mathrm{dist}\left(\mathbb{R}^d,D(\lambda)\right)
  \end{equation}
  with
  \begin{equation}\label{eq:dispersion} D(\lambda):=\{x\in \mathbb{C}^d:\,\,{\tilde P}(x,\lambda)=0\}
  \end{equation}
  for Theorem \ref{thm3:resolventdecay},
  we find by Lemma \ref{Lem:relationship analiticitydecay} that there exists $\varepsilon>\mathrm{dist}\left(\mathbb{R}^d,D(\lambda)\right)$ such that for each $l\in W$, the map $x\mapsto \hat{u}(x,l)$ is analytic in a strip $$\mathcal{S}_{\varepsilon}:=\{z\in \mathbb{C}:\,\,\max_{j=1,\ldots,\,d}\abs{\mathrm{Im}\,z_j}<\varepsilon\}.$$ In both cases condition \ref{condition2} of Theorem \ref{Kuch-Vainb theorem} holds with $\Omega'=\mathcal{S}_{\varepsilon}$ and $\Omega=\mathbb{C}^d$. Thus for each $l\in W$, $x\mapsto \hat{u}(x,l)$ is an entire function and once more this contradicts the weak form of unique continuation in \cite[P.49]{bk13} or \cite{LyMa18}.

  \end{proof}


  \section{Decay rate asymptotics: Proofs of Theorem \ref{thm3:resolventdecay} \ref{asympthmii} and \ref{prop:optimal}}
  \label{Sec:proofscontinued}
\subsection{Proof of Theorem \ref{thm3:resolventdecay} \ref{asympthmii}}
In view of Lemma \ref{prop:floquetTransf} in order to estimate $r(\lambda)$ given by \eqref{eq:defdistance} we consider the equation

  \begin{equation}
    \det(D^x+B-\lambda I)=0
  \end{equation}

with $D^x$ and $B$ given by Lemma \ref{prop:floquetTransf}. In particular, we recall that $B$ is independent of $\lambda$ and $$D^x(n,n')
=-\sum^d_{j=1}\left(\mu^{j}_{n} e^{2\pi i x_j}+\frac{1}{\mu^{j}_{n}e^{2\pi i x_j}}\right)\delta_{nn'}.$$ 
Denote by $\norm{B}$ the operator norm of $B$. 
An intermediate step for the proof of Theorem \ref{thm3:resolventdecay} \ref{asympthmii} which connects the desired decay rate to that of the discrete Laplacian is given by the Lemma below, which is essentially the Gershgorin circle theorem.
\begin{lem}\label{Lem:boundsolperiodiccase}
Let $\lambda\in \mathbb{C}$ and $x=x(\lambda)=(x_1(\lambda),\ldots,x_d(\lambda))$ be a solution of
\begin{equation}\label{eq:dispersion2}
    \det(D^x+B-\lambda I)=0.
  \end{equation} With $D^x$ and $B$ given by Lemma \ref{prop:floquetTransf}.
    There exists $n\in W$ such that
    \begin{equation}\label{eq:dispersionassymp}
    \abs{D^x(n,n)-\lambda}\leq\norm{B}.
    \end{equation}
\end{lem}

\begin{proof}
If $B\equiv 0$ the result is immediate so assume $B$ has at least one nonzero entry.
Assume for the sake of contradiction that $$\abs{D^x(n,n)-\lambda}>\norm{B}$$ for all $n\in W$. In this case
denoting $A=D^x-\lambda I$ we see that $A$ is a diagonal matrix which is invertible whenever $\norm{B}\neq 0$. Moreover
$$A+B=A(I+A^{-1}B)$$
where
$$\norm{A^{-1}B}\leq \frac{\norm{B}}{\min_{n\in W}{\abs{D^x(n,n)-\lambda}}}<1.$$
But in this case $(I+A^{-1}B)$ is invertible and so is $A+B$ which contradicts \eqref{eq:dispersion2}.

\end{proof}

\begin{lem}\label{Lem:freecaseassymp}
    Let 
    $$
    g(x_1,\ldots, x_d) = 2\sum_{j=1}^d \cos(2\pi x_j)
    $$
    and $$g^{-1}(\lambda)=\{x\in \mathbb{C}^d:\,\,g(z)=\lambda\}.$$
    Then 
    \begin{equation}\label{What we want to show...1}
    \lim_{|\lambda| \to \infty} \frac{2\pi \mathrm{dist}(g^{-1}(\lambda),\mathbb{R}^d)} {\ln|\lambda|} = 1.
    \end{equation}

\end{lem}
\begin{proof}
Notice that if $x=(x_1,\ldots,x_d)$
\begin{eqnarray*}
|g(x)|\leq 2\sum_{j = 1}^d e^{2\pi |\mathrm{Im}\,x_j|}.
\end{eqnarray*}
Thus, when $g(x_1,\ldots, x_d) = \lambda$, we have that 
\begin{equation}\label{Estimate on the max of the yjs}
\max_{1 \leq j \leq d}|\mathrm{Im}\,x_j| \geq (2\pi)^{-1}\ln{|\lambda| \over 2d}
\end{equation}
hence
\begin{equation}\label{the lower bound}
\mathrm{dist}(g^{-1}(\lambda),\mathbb{R}^d)=\min_{g(z) = \lambda} \sqrt{\sum_{j=1}^d |\mathrm{Im}\,z_j|^2} \geq (2\pi)^{-1}\ln{|\lambda| \over 2d}.
\end{equation}
In order to obtain an upper bound, consider the equation 
\begin{equation}\label{there is a solution with all coordinates the same}
    g(x_1, \frac{1}{4}, \ldots, \frac{1}{4}) = 2\cos(2\pi x_1) = \lambda
\end{equation}
and for $d = 1$, consider similarly $g(z) = \lambda$.
If $0 \leq x_1 < 1$ we consider the solution of
\begin{eqnarray*}
     2\cos(2\pi x_1) = {\lambda} 
\end{eqnarray*}
such that
\begin{eqnarray}
  \mathrm{Im}\,x_1 = (-2\pi)^{-1}\ln\left| {{\lambda }  + \sqrt{{\lambda^2 } - 4} \over 2}\right|. \label{the first solution we look at}
\end{eqnarray}
It readily follows that
\begin{equation}\label{The upper bound}
\limsup_{|\lambda|\to \infty}\frac{\min_{g(z) = \lambda} \sum_{j=1}^d \sqrt{|\mathrm{Im}\,x_j|^2}}{\ln |\lambda|} \leq  \frac{1}{2\pi}.
\end{equation}
\eqref{What we want to show...1} is now readily implied by \eqref{the lower bound} and \eqref{The upper bound}.
\end{proof}

Combining Lemmas \ref{Lem:boundsolperiodiccase} and \ref{Lem:freecaseassymp} we have

\begin{cor}
Let $r(\lambda)$ be given by \eqref{eq:defdistance}. Then
\begin{equation}\label{asympperiz}
\lim_{|\lambda|\to \infty}\frac{ r(\lambda)}{\ln|\lambda|}=1.
\end{equation}
\end{cor}

  \subsection{Asymptotic optimality of the bounds:~proof of Theorem \ref{thm3:resolventdecay}\ref{prop:optimal}}
Recall that we denote $H_0=-\Delta+V$.
   We now discuss the construction of solutions $u:\mathbb{Z}^d\rightarrow \mathbb{C}$ of the equation 
  \begin{equation}\label{eq:eigen}
  H_0u+vu=\lambda u
  \end{equation}
under the assumption that
  \begin{equation}\label{eq:decayv}
  \abs{v(n)}\leq C_{\alpha}e^{-\alpha\abs{n}}\,\,\text{for all}\,\, \alpha>0\,\, \text{and}\,\,n\in \mathbb{Z}^d.
  \end{equation}
  Letting $\gamma(\lambda)=\mathrm{dist}(\lambda,\sigma(H_0))$ these solutions are shown to have the property that
  \begin{equation}\label{eq:decayu}
  \lim_{n\to \infty}\frac{\ln \abs{u(n)}}{\abs{n}}<\mu\,\,\text{for any}\,\,\mu<\mu_0(\lambda,d):=\ln \gamma(\lambda) -\ln(2d).
  \end{equation}

\begin{prop}\label{prop:example}
 There exists a set $\mathcal{E}\subset \mathbb{R}$ of Lebesgue measure zero such that for all $\lambda \in \mathbb{C}\setminus(\mathcal{E}\cup \sigma(H_0))$ there are $v=v_{\lambda}:\mathbb{Z}^d\rightarrow \mathbb{C}$ and $u=u_{\lambda}:\mathbb{Z}^d\rightarrow \mathbb{C}$ satisfying \eqref{eq:decayv} and \eqref{eq:decayu}, respectively. Moreover, if $\lambda \in \mathbb{R}\setminus(\mathcal{E}\cup \sigma(H_0))$ then $v_{\lambda}$ is real-valued.

\end{prop}

\begin{remark}
    In particular, this implies that Theorem \ref{thm3:resolventdecay} is asymptotically sharp in the sense that by taking $\lambda$ sufficiently large one may find solutions to \eqref{eq:eigen} which  decay exponentially with $|n|$ at a rate which approaches $\ln |\lambda|$ as $|\lambda|\to \infty$ along the real axis.
\end{remark}

\begin{proof}[Proof of Proposition \ref{prop:example}]
Throughout this proof we will make use of a few standard facts and notation, which we now recall. Given $\lambda \notin \sigma(H_0)$ we write $\gamma(\lambda):=\mathrm{dist}(\lambda,\sigma(H_0))>0$ and denote by 
  \begin{equation}\label{eq:Greendfn}
  G_0(m,n,\lambda):=\langle \delta_{m},(H_0-\lambda I)^{-1}\delta_n\rangle
  \end{equation} the Green's function of $H_0$ at $\lambda.$ The functions $\{\delta_l\}_{l\in \mathbb{Z}^d}$ are defined as
  $\delta_l(n)=\delta_{nl}$ for all $n\in \mathbb{Z}^d$, where $\delta_{nl}$ is the Kronecker delta.
  The Combes-Thomas bound \cite[Theorem 10.5]{A-W-B}, when specialized to our context, implies that
  
  \begin{equation}\label{eq:CT}
  \abs{G_0(m,n,\lambda)}\leq \frac{1}{\gamma(\lambda)-S_{\mu}}e^{-\mu\abs{m-n}}
  \end{equation} for all $\mu$ such that $S_{\mu}:=2de^{\mu}<\gamma(\lambda)$. That is, equation \eqref{eq:CT} holds whenever $\mu<\ln(\gamma(\lambda))-\ln(2d)$.
  Finally, we write for any $\lambda \notin \sigma(H_0)$
  \begin{equation}\label{eq:defresolventu}
  u=u_{\lambda}:=(H_0-\lambda)^{-1}\delta_0.
  \end{equation}


Let 
\begin{equation}\label{glast}
   \mathcal{E}=\{\lambda\in \R\setminus\sigma(H_0):  G_0(0,0;\lambda)=0\}.
\end{equation}

We are going to prove that $\mathcal{E}$ has Lebesgue measure zero. For that purpose, we make use of Herglotz theory. By \eqref{eq:symmetry}, the map 
  \begin{equation}\label{eq:Herglotzmap}
  z\mapsto G_0(z):= G_0(0,0;z)
  \end{equation}
   defines a Herglotz-Nevalinna function. Namely:
   \begin{enumerate}[label=(\roman*)]
       \item $\mathrm{Im}\,G_0(z) >0$ whenever $\mathrm{Im(z)}>0$.
       \item $G_0:\mathbb{C}^{+} \rightarrow \mathbb{C}^{+}$ is analytic.
   \end{enumerate}

   Given any Herglotz function $F:\mathbb{C}^{+}\rightarrow \mathbb{C}^{+}$, for Lebesgue almost every $\mathrm{Re}\lambda \in \mathbb{R}$ the boundary value $F(\mathrm{Re}\lambda+i0):=\lim_{\mathrm{Im}\lambda\to 0^{+}} F(\lambda)$ exists and is finite by a theorem of de la Vallé-Poussin, see \cite[Corollary 3.29]{GTES}. In particular, applying this fact to the Herglotz function  $F(z)=-\frac{1}{G_0(0,0,z)}$  and noting that $G_0(0,0,z)$ is a continuous function for $z$ in $\C\setminus \sigma(H_0)$, we conclude that $\mathcal{E}$ has Lebesgue measure zero.

   Since $H_0$ is self-adjoint,   one has that
  \begin{equation}\label{eq:symmetry}
  \mathrm{Im}(G_0(0,0;\lambda))=2\mathrm{Im(\lambda)}\abs{(H_0-\bar{\lambda})^{-1}\delta_0}^{2}. 
  \end{equation}
  Therefore,  $u(0)=G_0(0,0;\lambda)\neq 0$ for any $\lambda\in \mathbb{C}\setminus \mathbb{R}$.
  By \eqref{glast} and \eqref{eq:symmetry}, the function 
      $v=-\frac{1}{G_0(0,0;\lambda)}\delta_0$
       is well-defined for all $\lambda\in \mathbb{C}\setminus (\mathcal{E}\cup \sigma(H_0))$ and we have that 
      \begin{equation}
        \abs{v(n)}\leq \frac{1}{\abs{u_\lambda(0)}}e^{-\beta \abs{n}}\,\,\text{for any}\,\,\beta>0.  
      \end{equation}
 It is immediate to check that
  \begin{equation}\label{eq:perturbedzeq}
  -\Delta u+Vu+vu=\lambda u.
  \end{equation}
  Moreover, the Combes-Thomas estimate ensures that \begin{equation}\abs{u(n)}\leq \frac{1}{\gamma(\lambda)-S_{\mu}}e^{-\mu\abs{n}},\,\,\,\text{for any}\,\, \mu< \ln(\gamma(z))-\ln(2d).
  \end{equation}
  
 Finally, 
 by \eqref{eq:symmetry},
 for $\mathrm{Re}\lambda \in \mathbb{R}\setminus(\sigma(H_0)\cup \mathcal{E})$ we have that $\mathrm{Im} G(\mathrm{Re}\lambda+i0)=0$ and thus the function $v_{\lambda}$ is real-valued.

\end{proof}



\appendix\label{The appendix}

\section{Auxiliary results in Complex analysis and Proof of Lemma \ref{Lem:vanishing}}\label{sec:proofovanishing}

\subsection{Real-analytic coordinate changes}

We will need real-analytic versions of the inverse function theorem and related change of coordinate theorems. As these results are well-known to experts and their proofs closely follow their holomorphic counterparts we only provide the statements.


\begin{lem}\label{rearranging jacobian lemma}
Suppose that $U \subset \mathbb{R}^d$ is a neighborhood of a point $p$ such that $u:U \to \mathbb{R}^d$ is a real-analytic map with an invertible Jacobian matrix at $p$.  Then there is a neighborhood $V \subset \mathbb{C}^d$ of $p$ and an injective holomorphic map $f:V \to \mathbb{C}^d$ such that $f|_{U \cap V} = u$.
\end{lem}


\begin{lem}(Real-analytic inverse function theorem)\label{Real-analytic inverse function theorem} Suppose $u:\mathbb{R}^d \to \mathbb{R}^d$ is a real-analytic map such that at a point $p \in \mathbb{R}^d$, the Jacobian determinant $J_{\mathbb{R}}u(p) \neq 0$.  Then there are neighborhoods $U$ of $p$ and $V$ of $F(p)$ such that $u:U \to V$ is bijective and $u^{-1}|_V$ is real-analytic.
    
\end{lem}


    

The following lemma follows by repeating \cite[Chap. I Sec. B, Theorem 9]{GuRo65} \textit{mutatis mutandi} for the specific case of a codimension one submanifold and real-analytic replacing instances of holomorphic or complex.  
\begin{lem}\label{change of coordinates for real-analytic manifold lemma}
Suppose $d > 1$. Let $M = \{x \in \mathbb{R}^d: f(x) = 0\}$ be a $(d-1)$-dimensional real-analytic manifold containing a point $p$.  There is a bijective, real-analytic map $F:U \to V$  between neighborhoods $U, V \subset \mathbb{R}^d$ containing $p$ and 0 respectively such that $F(p) = 0$ and 
$$
F(M \cap U) = \{(z^{\prime}, z_d) \in V: z_d = 0\}.
$$
\end{lem}

Lemmas \ref{rearranging jacobian lemma} and \ref{change of coordinates for real-analytic manifold lemma} together give
\begin{lem}\label{special real-analytic and biholomorphic isomorphism}
Suppose $d > 1$. Let $M \subset \mathbb{R}^d$ be a $(d - 1)$-dimensional real-analytic manifold containing $p$.  For all $\epsilon > 0$ sufficiently small, there is a neighborhood $U \subset \mathbb{C}^d$ containing $p$, a neighborhood $V = (-\epsilon, \epsilon)^d \oplus i(-\epsilon, \epsilon)^d$ and a biholomorphism $F: U \to V$ such that
$$
F(p) = 0, \quad F(U \cap M) = \{(z^{\prime}, z_d):\, z^{\prime} \in \mathbb{R}^{d-1},\, z_d = 0\} \cap V, \quad F(U \cap \mathbb{R}^d) = V \cap \mathbb{R}^d.
$$
Moreover, there are constants $c_1$ and $c_2$ such that $0 < c_1 \leq |\det J_\mathbb{C}F|^2 \leq c_2 < \infty$, where $J_{\mathbb{C}}F$ denotes the Jacobian matrix of $F$.
\end{lem}
\begin{proof}
    By Lemma \ref{change of coordinates for real-analytic manifold lemma}, there is a real-analytic isomorphism $F^{\prime}:U^{\prime} \to V^{\prime}$ between neighborhoods $U^{\prime}$, $V^{\prime} \subset \mathbb{R}^d$ containing $p$ and 0 respectively such that
    $$
    F^{\prime}(p) = 0 \quad F^{\prime}(U^{\prime} \cap M) = \{(z^{\prime}, z_d) \in V^{\prime}:\, z_d = 0\}.
    $$
    By Lemma \ref{rearranging jacobian lemma}, there is a biholomorphic map between neighborhoods $U, V \subset \mathbb{C}^d$ containing $p$ and 0 respectively such that $F|_{U \cap U^{\prime}} = F^{\prime}|_{U \cap U^{\prime}}$.  Hence $F(p) = 0$.  By shrinking $U$ to a sufficiently small ball, we may suppose that 
    $
    U \cap \mathbb{R}^d \subset U^{\prime},
    $
    and the constants $c_1$ and $c_2$ exist.
    Since $F$ agrees with $F^{\prime}$ on their common domain, $F(U \cap \mathbb{R}^d) \subset F(U) \cap \mathbb{R}^d$. Since $F(z) = \overline{F}(\bar{z})$ on $U \cap \mathbb{R}^d$, the equality holds on $U$ as well.  If for some $q \in U\setminus\mathbb{R}^d$, it was the case that $F(q) \subset \mathbb{R}^d$, then $F(q) = F(\bar{q})$, violating that $F$ is bijective. Thus, 
    $$
    F(U \cap \mathbb{R}^d) = F(U) \cap \mathbb{R}^d = V \cap \mathbb{R}^d.
    $$
    It now follows that $F(U \cap M) = \{(z^{\prime}, z_d):\, z^{\prime} \in \mathbb{R}^{d-1},\, z_d = 0\} \cap V$. By shrinking $V$, and consequently $U = F^{-1}(V)$ may no longer be a ball, $V = (-\epsilon, \epsilon)^d \oplus i(-\epsilon, \epsilon)^d$.  
\end{proof}
\subsection{Lemmas on analytic subsets}
Below $\Omega$ will denote a domain in $\mathbb{C}^d$. Additionally, $z = (z_1, \ldots, z_d) = (z^{\prime}, z_d)$ and $z_j = x_j + iy_j$ for $j = 1,\ldots, d$. The function $J$ will denote the complex structure map on $\mathbb{C}^d$; that is, for each $q \in \mathbb{R}^{2d}$, $J:T_q(\mathbb{R}^{2d}) \to T_{q}(\mathbb{R}^{2d})$ by
$$
J\left({\partial \over \partial x_i}\Big{|}_{q}\right) = {\partial \over \partial y_i}\Big{|}_{q}, \quad J\left({\partial \over \partial y_i}\Big{|}_{q}\right) = -{\partial \over \partial x_i}\Big{|}_{q}.
$$

\begin{lem}\label{lemma with conclusion no open subset}
Suppose $d > 1$. Let $A^{\prime}$ be an analytic subset of $\Omega$ of (complex) dimension less than $d - 1$.  Let $M$ be a real manifold in $\mathbb{R}^d$ of (real) dimension $d - 1$.  Then there is no open subset of $M$ contained in $A^{\prime}$.  
\end{lem}

\begin{proof}
    Towards a contradiction, suppose that there is an open subset $O \subset M$ contained in $A^{\prime}$. Define $B_0 = A^{\prime}$ and for $k \geq 1$,
$$
B_k = \hbox{sng}(\cdots(\hbox{sng}(A^{\prime})))
$$
where $\hbox{sng}$ is repeated $k \geq 1$ times.  Notice that
\begin{equation}\label{decomposition of A prime}
A^{\prime} = \sqcup_{k=0}^N \reg(B_k) \sqcup B_{N + 1}, \quad \forall N,
\end{equation}
where $\sqcup$ denotes the disjoint union and some of the sets in the union may be empty.
Recall that if an analytic subset $A$ has dimension $j$, then $\sng(A)$ has dimension at most $j-1$.  Moreover, if $A$ is zero dimensional, then $A = \reg(A)$.  So the (complex) dimension of $B_k$ is less than $d-1$ and there is a minimal number $j$ such that $O \cap \reg(B_j) \neq \emptyset$.  Since $j$ is minimal, by \eqref{decomposition of A prime} $O \subset B_j$.  Let $p \in O \cap \reg(B_j)$.  Both $M$ and $B_j$ have subspace topologies inherited from $\Omega$.  So for all neighborhoods $U$ of $p$ in $\mathbb{C}^n$ sufficiently small, 
\begin{equation}\label{C1}
C_1 := U \cap M \subset O,
\end{equation}
is $(d - 1)$-dimensional, and 
\begin{equation}\label{C2}
C_2 := U \cap B_j
\end{equation}
is a complex manifold of dimension strictly less than $d - 1$.  
Since $T_pC_1 \subset \hbox{span}\{{\partial \over \partial x_i}|_{p}\}_{i=1}^d$ for all $p \in C_1$, $C_1$ is a totally real submanifold; that is 
$$
T_pC_1 \cap JT_pC_1 = \{0\}.
$$
Since $C_2$ is a complex manifold, $T_pC_2 = JT_pC_2$ for all $p \in C_2$.  Since $O \subset B_j$,  by \eqref{C1} and \eqref{C2}, $C_1 \subset C_2$. Let $\iota:C_1 \to C_2$ denote the inclusion map. For any $p$, the maximal (real) dimension of the totally real subspace $\iota_*T_pC_1$ of $T_pC_2$ is $\dim_{\mathbb{R}}(T_pC_2)/2 < d - 1$.  This contradicts that the (real) dimension of $C_1$ is $d - 1$.
\end{proof}

\begin{definition}\label{Definition in the appendix of what the order is and at least order}
A function $f \in C^{k}(\Omega)$ vanishes to order at least $k$ at $z_0 \in \Omega$ if 
$$
{\partial^{\alpha}f \over \partial z^{\alpha}}(z_0) = 0, \quad |\alpha| < k,
$$
and it vanishes to order (exactly) $k$ if moreover there exists $\beta \in \mathbb{N}^d$ with $|\beta| = k$ such that ${\partial^{\beta} f \over \partial z^{\beta}}(z_0) \neq 0$. Additionally, $f$ vanishes to order at least $k$ on a set $S$ if it vanishes to order at least $k$ at all points of the set $S$.
\end{definition}
\begin{lem}\label{lemma involving weierstrass division theorem}
Suppose $d > 1$. Suppose $f$ is holomorphic in a neighborhood $U$ of $0 \in \mathbb{C}^d$ and $f$ vanishes to order at least $k \geq 1$ at all points of $U \cap \left(\mathbb{R}^{d-1}\times \{0\}\right)$.  Then $f(z) = z_d^kq(z)$ where $q$ is a holomorphic function in a neighborhood $V_1 \subset U$.  Moreover, if $f$ vanishes to order (exactly) $k$ at $0\in \mathbb{C}^d$, 
 then $q \neq 0$ in a neighborhood $V_2 \subset V_1$ of 0.
\end{lem}

\begin{proof}
By the Weierstrass division theorem \cite[Chap. II Sec. B, Theorem 3]{GuRo65}, in a small neighborhood $V_1 \subset U$ of 0, 
\begin{equation}\label{First application of Weierstrass Division Theorem}
f(z) = z_d^kq(z^{\prime}, z_d) + r(z^{\prime}, z_d),
\end{equation}
where $r$ is a polynomial of $z_d$ of degree less than $k$ with holomorphic coefficients in the $z^{\prime}$ variables.  Since $f$ vanishes to order at least $k$ at $(x^{\prime}, 0)$ for any $x^{\prime} \in \mathbb{R}^{d-1}$ near $0^{\prime}$,
$$
0 = {\partial^{\alpha} f \over \partial z^{\alpha}}(x^{\prime}, 0) = {\partial^{\alpha} r \over \partial z^{\alpha}}(x^{\prime}, 0), \quad |\alpha| < k,
$$
which implies that $r \equiv 0$.  Equation \eqref{First application of Weierstrass Division Theorem} now implies that
\begin{equation}\label{derivatives of equation of order equal to k}
0 = {\partial^{\alpha} f \over \partial z^\alpha}(x^{\prime}, 0), \quad |\alpha| = k, \quad  \alpha_d< k.
\end{equation}
If $f$ vanishes to order $k$ at $0\in \mathbb{C}^d$, \eqref{derivatives of equation of order equal to k} implies that 
$$
k!q(0, 0) = {\partial^k f \over \partial z_d^k}(0, 0) \neq 0.
$$ By continuity, $q$ is nonzero in a neighborhood $V_2 \subset V_1$ of 0.  
\end{proof}
Recall from Section \ref{Section: Kuchvainb} that every analytic subset $A$ has a unique decomposition into irreducible components $A = \cup_{j \in I}Z_j$ where $Z_j$ is irreducible. The union is locally finite, which means that for all compact sets $K \subset \mathbb{C}^d$,
$$
\left| \{j \in I: K \cap Z_j \neq \emptyset\} \right| < \infty.
$$
For any two distinct irreducible components $Z_i$ and $Z_j$, $Z_i \cap Z_j \subset \sng(A)$, cf. \cite[p.5.1.~top and p.5.4.~Theorem]{Ch89}.
\begin{lem}\label{lemma that involves local finiteness and irreducibility}
Let $A$ be an analytic subset of $\Omega$ and $A = \cup_{j \in I}Z_j$ be its decomposition into irreducible components. For any $j \in I$ and $p \in \reg(A) \cap Z_j$, for any sufficiently small neighborhood $U$ of $p$, 
$$
U \cap A = U \cap \reg(A) = U \cap Z_j,
$$
and $U \cap A$ is a complex manifold.
\end{lem}
\begin{proof}
    For all neighborhoods $U$ of $p$ sufficiently small, that 
    $$
    U \cap A = U \cap \reg(A),
    $$
    follows from the definition of $\reg(A)$.  Fix such a neighborhood $U_1$.  By the local finiteness of the decomposition of $A$ into irreducible components, there are finitely many indices $i_1, \ldots, i_N \in I$ such that
    $$
    Z_{i_k} \cap \overline{U_1} \cap A \neq \emptyset.
    $$
    Let $Z_{i_k} \neq Z_j$.  If for all neighborhoods $U_2$ of $p$,
    $$
    Z_{i_k} \cap U_2 \cap A \neq \emptyset,
    $$
    then $p \in \overline{Z_{i_k}} = Z_{i_k}$.  Since $p \in \reg(A)$ and $Z_{i_k} \cap Z_j \subset \sng(A)$, this is impossible.  Thus, for a neighborhood $U_{i_k}$ sufficiently small,
    $$
    Z_{i_k} \cap U_{i_k} \cap A = \emptyset.
    $$
    Let $U = U_1 \cap \bigcap\limits_{i_k \neq j} U_{i_k}$.  Then 
    $$
    U \cap A = U \cap Z_j.
    $$
\end{proof}

\begin{lem}\label{C1 manifold leads to real-analytic manifold lemma}
Suppose $d > 1$. Let $A$ be a $(d - 1)$-dimensional analytic subset of $\Omega$.  
 Suppose that $A \cap \mathbb{R}^d$ has Hausdorff dimension $d-1$. Then $A \cap \mathbb{R}^d$ contains a real-analytic $(d-1)$-dimensional manifold.
\end{lem}
\begin{proof}
    Every real-analytic subset of is a locally finite (hence, countable) union of pairwise disjoint smooth real-analytic submanifolds. Therefore, if $A\cap\R$ has Hausdorff dimension $d-1$,  $A \cap \mathbb{R}^d$ contains a real-analytic $(d-1)$-dimensional manifold. 
    See \cite[Theorem 1.2.10]{Trotman}.
\end{proof}


\subsection{Proof of Lemma \ref{Lem:vanishing}}  
\begin{proof}[Proof of Lemma \ref{Lem:vanishing}]

Fix $j$.  Recall from \eqref{definition of Aj prime}
\begin{equation*}
    A_j^{\prime} = \left\{z \in \Omega: {\partial^{\alpha}h \over \partial z^{\alpha}}(z) = 0, \quad 0 \leq |\alpha| \leq m_j\right\} \cap Z_j,
\end{equation*}
where $m_j$ is the minimal order of the zeros of $h$ on $Z_j$.  For clarity, we will suppress the subscript $j$ in our notation and let 
$
Z = Z_j, A^{\prime} = A^{\prime}_j, 
$ and $m = m_j$. We consider case \ref{condition1}. 

\begin{claim*}
    There is a real-analytic $(d - 1)$-dimensional manifold $M \subset Z \cap \mathbb{R}^d$ which contains only regular points of $A$ and $h$ vanishes to order $m$ on $M$.
\end{claim*}
\begin{proof}
   By Lemma \ref{C1 manifold leads to real-analytic manifold lemma}, we may assume without loss of generality that $Z\cap \mathbb{R}^d$ contains a real-analytic manifold of dimension $d-1$, which we will call $\widehat{M}$. By Lemma \ref{lemma with conclusion no open subset} and \eqref{dimension of Ajprime}, $\widehat{M} \not\subset A^{\prime}$. Let
    \begin{equation}\label{long equation defining O2}
    M := \widehat{M} \cap \left(Z\setminus A^{\prime}\right) \cap \reg(A).
    \end{equation}
    $\widehat{M} \cap (Z\setminus A^{\prime})$ is open in $\widehat{M}$ because $Z \setminus A^{\prime}$ is open in $Z$, $\widehat{M}$ and $Z$ both have the subspace topology of $\mathbb{C}^d$ and $\widehat{M} \subset Z$.  By Lemma \ref{lemma with conclusion no open subset}, $\widehat{M} \cap (Z \setminus A^{\prime})$ cannot be contained in $\sng(A)$.  So $M$ is nonempty. By a similar argument, $M$ is open in $\widehat{M}$. Additionally from \eqref{long equation defining O2}, $h$ vanishes to order $m$ on $M$ and $M \subset \reg(A)$. 
    \renewcommand{\qedsymbol}{$\blacksquare$}
\end{proof}

Without loss of generality, assume $0 \in M$ and let $F: U \to V = (-\epsilon, \epsilon)^d \oplus i(-\epsilon, \epsilon)^d$ be as in Lemma \ref{special real-analytic and biholomorphic isomorphism} with $p = 0$.  As stated in that lemma, $\epsilon > 0$ is a sufficiently small quantity. By Lemma \ref{lemma that involves local finiteness and irreducibility}, after shrinking $\epsilon$, $U \cap  A = U \cap Z$. In the sequel, we will henceforth use the convention that $\epsilon > 0$ may become a smaller number at each appearance. Since $h \circ F^{-1} \equiv 0$ on $(-\epsilon, \epsilon)^{d-1} \times \{0\}$, 
$$
h \circ F^{-1}(\xi) \equiv 0, \quad \xi \in \left((-\epsilon, \epsilon)^{d-1} \oplus i(-\epsilon, \epsilon)^{d-1}\right) \times \{0\} = \left(\mathbb{C}^{d-1} \times \{0\}\right) \cap V
$$
Since $h$ vanishes to order $m$ on $M$, $h \circ F^{-1}$ vanishes to order $m$ on $(-\epsilon, \epsilon)^{d-1} \times \{0\}$.  
By Lemma \ref{lemma involving weierstrass division theorem}, 
  \begin{equation*}
  h \circ F^{-1}(\xi) = \xi_d^{m}q_1(\xi), \quad \xi \in V,
  \end{equation*}
  where $q_1 \neq 0$. This shows that
  \begin{equation}\label{where h vanishes}
  \{\xi \in V:\, h \circ F^{-1}(\xi) = 0\} = \left(\mathbb{C}^{d-1} \times \{0\}\right) \cap V.
  \end{equation}

Below $c$ will be a positive numerical constant which may be smaller at each appearance. Since $f \in L^2_{loc}(\mathbb{R}^d)$,
    \begin{eqnarray}
    \infty &>& \int_{U \cap \mathbb{R}^d} |f(x)|^2 dx_1\cdots dx_d \nonumber \\
    &=& \int_{F(U \cap \mathbb{R}^d)} |f \circ F^{-1}(\xi_1,\ldots, \xi_d)|^2|\det J_{\mathbb{R}}F^{-1}|d\xi_1\cdots d\xi_d \nonumber
    \\
    &\geq & c\int_{(-\epsilon, \epsilon)^d}|f \circ F^{-1}(\xi_1,\ldots, \xi_d)|^2 d\xi_1\cdots d\xi_d \nonumber
    \\
    &\geq& c\int_{(-\epsilon, \epsilon)^{d-1}}\int_{-\epsilon}^{\epsilon} {|g \circ F^{-1}(\xi_1,\ldots, \xi_d)|^2 \over |\xi_d|^{2m}}d\xi_d\, d(\xi_1\cdots \xi_{d-1}),\label{now we will show the numerator vanishes}
    \end{eqnarray}
    where the first equality used that the real Jacobian matrix of a holomorphic mapping satisfies  $\det J_\mathbb{R}F = |\det J_{\mathbb{C}}F|^2$ and $c$ is guaranteed to exist by Lemma \ref{special real-analytic and biholomorphic isomorphism}.  If $g \circ F^{-1}(\xi^{\prime}, 0) \not\equiv 0$, then \eqref{now we will show the numerator vanishes} is impossible.  Thus,
    $$
    g\circ F^{-1}(\xi^{\prime}, 0) \equiv 0, \quad \xi^{\prime} \in (-\epsilon, \epsilon)^{d-1}.
    $$
     It follows that $g \circ F^{-1} \equiv 0$, in $(\mathbb{C}^{d-1}\times\{0\}) \cap V.$
\begin{claim*}
    After shrinking $V$, $g \circ F^{-1}$ vanishes to order at least $k \geq m$ on $\left(\mathbb{C}^{d-1} \times \{0\}\right) \cap V$.
\end{claim*}
\begin{proof}
    Suppose the minimum of the vanishing orders of $g \circ F^{-1}$ on $(-\epsilon, \epsilon)^{d-1}\times\{0\}$ is $k$ and is achieved at $(x_0^{\prime}, 0)$. Since the order is minimal, $g \circ F^{-1}$ vanishes to order $k$ in a small neighborhood of $(x_0^{\prime}, 0)$ in $(-\epsilon, \epsilon)^{d-1} \times \{0\}$.    By Lemma  \ref{lemma involving weierstrass division theorem}, for some small $\epsilon^{\prime} > 0$ and $S =  \prod_{j=1}^{d-1}((x_0^{\prime})_j - \epsilon^{\prime}, (x_0^{\prime})_j + \epsilon^{\prime})$,
    \begin{equation}\label{weierstrass division for g of F inverse}
    g \circ F^{-1}(\xi) = \xi_d^kq_2(\xi - (x_0^{\prime}, 0)), \quad \xi \in S \times (-\epsilon^{\prime}, \epsilon^{\prime}) \subset (-\epsilon, \epsilon)^d.
    \end{equation}
    where $q_2 \neq 0$.  Then \eqref{now we will show the numerator vanishes} implies that
    \begin{eqnarray*}
    \infty &>& c\int_{(-\epsilon, \epsilon)^{d-1}}\int_{-\epsilon}^{\epsilon} {|g \circ F^{-1}(\xi_1,\ldots, \xi_d)|^2 \over |\xi_d|^{2m}}d\xi_1\cdots d\xi_d 
    \\
    &>& c\int_{S}\int_{-\epsilon^{\prime}}^{\epsilon^{\prime}} {|\xi_d|^{2k} \over |\xi_d|^{2m}} d\xi_d\, d(\xi_1\cdots \xi_{d-1}).
    \end{eqnarray*}
    Thus, $k \geq m$. By the minimality of $k$, $g \circ F^{-1}$ vanishes to order at least $k$ which is at least $m$ on $(-\epsilon, \epsilon)^{d-1} \times \{0\}$.  Using Lemma \ref{lemma involving weierstrass division theorem}, $g \circ F^{-1}$ vanishes to order at least $k$ in a neighborhood of the origin in $\left(\mathbb{C}^{d - 1} \times \{0\}\right) \cap V$.  After shrinking $V$, the claim follows.
    \renewcommand{\qedsymbol}{$\blacksquare$}
    \end{proof}
    
    By \eqref{where h vanishes}, the zero set of $g \circ F^{-1}$ contains a small neighborhood of the origin intersected with the zero set of $h \circ F^{-1}$.  Thus, $g$ vanishes to order at least $m$ on a small open subset of the zero set of $h$.  Since $A$ (the zero set of $h$) equals Z near the origin, the proof of \ref{condition1} is complete.
    \medskip

    For case \ref{condition2}, since $Z \setminus A^{\prime}$ and $\reg(A)$ are open and dense in $Z$ and $Z \cap \Omega^{\prime}$ is a nonempty open subset of $Z$, $Z\setminus A^\prime$  contains a point $p \in \reg(A) \cap \Omega^{\prime}$ such that $h$ vanishes to order $m$ at $p$ and the dimension of $A$ at $p$ is $d - 1$.  By Lemma \ref{lemma that involves local finiteness and irreducibility}, there is a neighborhood $U \subset \Omega^\prime$ of $p$ such that
    $$
    U \cap A = U \cap \reg(A) = U \cap Z
    $$
    is a complex manifold.  After a holomorphic change of coordinates, without loss of generality,
    $$
    Z \cap U = \{(z_1, \ldots, z_d) \in U: z_d = 0\}.
    $$
    Since $m$ is minimal we can suppose that $h$ vanishes to order $m$ at all points of $Z \cap U$.  By Lemma \ref{lemma involving weierstrass division theorem}, $h(z) = z_d^{m}q_1(z)$ where $q_1$ is a nonzero holomorphic function.  Since $f \in \mathcal{O}(U)$, $g$ must vanish on the zero set of $h$ in $U$.  By Lemma \ref{lemma involving weierstrass division theorem} again, $g(z) = z_d^kq_2(z)$ for some $k$ and holomorphic function $q_2$ not identically zero on $Z \cap U$.  Since $f \in \mathcal{O}(U)$, $k \geq m$; that is, $g$ vanishes to order at least $k$, which is at least $m$ at all points of $U \cap Z$.


\end{proof}

\section*{Acknowledgments}

W.\ L.\  is  partially supported by  the Simons Foundation, NSF  DMS-2000345, NSF DMS-2052572 and NSF DMS-2246031. R.M. is partially supported by CNPq Universal grant 402952/2023-5. J.N.T. is partially supported by NSF DMS-2247175 Subaward M2401689.
R.M. and J.N.T.  would also like to thank Department of Mathematics, Texas A\&M University, where part of the work was done when they were visiting assistant professors there.

\section*{Statements and Declarations}
{\bf Conflict of Interest} 
The authors declare no conflicts of interest.

\vspace{0.2in}
{\bf Data Availability}
Data sharing is not applicable to this article as no new data were created or analyzed in this study.

\bibliographystyle{amsplain}

\bibliography{LMT}

\end{document}